\documentclass[a4paper]{amsart}

\usepackage{amsmath,amssymb,amsthm}
\usepackage{fullpage}
\usepackage[colorlinks,linkcolor=blue,urlcolor=blue,citecolor=blue,destlabel=true]{hyperref}
\usepackage{mathtools}
\usepackage[dvipsnames]{xcolor}
\definecolor{Red}{rgb}{1,0,0}

\usepackage{tikz}
\usetikzlibrary{matrix,arrows}

\theoremstyle{plain}
\newtheorem{lemma}{Lemma}[section]
\newtheorem{theorem}[lemma]{Theorem}

\newtheorem{corollary}[lemma]{Corollary}
\newtheorem{proposition}[lemma]{Proposition}
\newtheorem{conjecture}[lemma]{Conjecture}

\theoremstyle{definition}
\newtheorem{definition}[lemma]{Definition}
\newtheorem{example}[lemma]{Example}
\newtheorem{question}[lemma]{Question}

\newcommand{\N}{\mathbb{N}}
\newcommand{\Q}{\mathbb{Q}}
\newcommand{\R}{\mathbb{R}}
\newcommand{\RP}{\mathbb{R}\mathrm{P}}
\newcommand{\Z}{\mathbb{Z}}
\newcommand{\vr}[2]{\mathrm{VR}(#1;#2)}
\newcommand{\vrleq}[2]{\mathrm{VR}_\leq(#1;#2)}
\newcommand{\vrless}[2]{\mathrm{VR}_<(#1;#2)}
\newcommand{\cech}[2]{\mathrm{\check{C}}(#1;#2)}
\newcommand{\cechless}[2]{\mathrm{\check{C}}_<(#1;#2)}
\newcommand{\cechleq}[2]{\mathrm{\check{C}}_\le(#1;#2)}

\newcommand{\cechmleq}[2]{\mathrm{\check{C}}^m_\leq(#1;#2)}
\newcommand{\cechmless}[2]{\mathrm{\check{C}}^m_<(#1;#2)}
\newcommand{\vrm}[2]{\mathrm{VR}^m(#1;#2)}
\newcommand{\vrmleq}[2]{\mathrm{VR}^m_\leq(#1;#2)}
\newcommand{\vrmless}[2]{\mathrm{VR}^m_<(#1;#2)}

\newcommand{\diam}{\mathrm{diam}}

\newcommand{\so}{\mathrm{SO}}

\newcommand{\Gr}{\mathrm{Gr}}

\begin{document}

\title{Metric thickenings and group actions}
\author{Henry Adams}
\email{adams@math.colostate.edu}
\author{Mark Heim}
\email{mark.heim@colostate.edu}
\author{Chris Peterson}
\email{peterson@math.colostate.edu}

\begin{abstract}
Let $G$ be a group acting properly and by isometries on a metric space $X$; it follows that the \emph{quotient} or \emph{orbit} space $X/G$ is also a metric space.
We study the Vietoris--Rips and \v{C}ech complexes of $X/G$.
Whereas (co)homology theories for metric spaces let the scale parameter of a Vietoris--Rips or \v{C}ech complex go to zero, and whereas geometric group theory requires the scale parameter to be sufficiently large, we instead consider intermediate scale parameters (neither tending to zero nor to infinity).
As a particular case, we study the Vietoris--Rips and \v{C}ech thickenings of projective spaces at the first scale parameter where the homotopy type changes.
\end{abstract}


\maketitle

\section{Introduction}

Vietoris--Rips and \v{C}ech complexes are geometric constructions which transform a metric space $X$ into a simplicial complex depending on the choice of a scale parameter $r$.
Indeed, the Vietoris--Rips complex $\vr{X}{r}$ includes as its simplices all finite subsets of $X$ of diameter at most $r$, and the \v{C}ech complex $\cech{X}{r}$ includes all finite subsets of $X$ contained in a ball of radius $r$.
These complexes have been used in nerve lemmas~\cite{Borsuk1948} to relate homotopy types of spaces with good covers thereof.
They have also been used to define (co)homology theories for metric spaces~\cite{Hausmann1995,lefschetz1942algebraic,Vietoris27}.
Indeed, one can associate to a metric space the homology or cohomology of its Vietoris--Rips or \v{C}ech simplicial complex and then take the limit as the positive scale parameter goes to zero.

Vietoris--Rips complexes were independently developed for use in geometric group theory as a way to thicken a metric space, i.e.\ to view it from a zoomed-out perspective~\cite{Gromov}.
In particular, one can use Vietoris--Rips complexes to construct finite-dimensional Eilenberg--MacLane spaces for torsion-free hyperbolic groups (Theorem~3.21 of~\cite{bridson2011metric}).
Indeed, let $G$ be a hyperbolic group, equipped with the shortest path metric on its Cayley graph for some choice of generators.
Then $\vr{G}{r}$ is contractible for scale $r$ sufficiently large, $G$ acts simplicially, and if $G$ is torsion free, then this produces a finite-dimensional model $\vr{G}{r}/G$ for the  Eilenberg--MacLane space $K(G,1)$.
Vietoris--Rips complexes have also been connected to Bestvina--Brady Morse theory~\cite{zaremsky2019}, singular homology theories depending on a choice of scale~\cite{goldfarb2016singular}, notions of homotopy type depending on a choice of scale~\cite{brodskiy2007rips,berestovskii2007uniform}, Borsuk--Ulam theorems into higher-dimensional codomains~\cite{ABF}, and to the filling radius in quantitative topology~\cite{lim2020vietoris,gromov1983filling,katz1991neighborhoods}.

More recently, in applied and computational topology, Vietoris--Rips and \v{C}ech complexes have been used to recover the ``shape'' of a dataset.
Indeed, there are theoretical guarantees that if $X$ is a sufficiently nice sample from an unknown underlying space $M$, then one can recover the homotopy types, homology groups, or approximate persistent homology of $M$ from $X$~\cite{ChazalDeSilvaOudot2013,ChazalOudot2008,Latschev2001,virk2019rips}.
In data analysis contexts, instead of letting $r$ be arbitrarily small (as for (co)homology theories), and instead of letting $r$ be sufficiently large (as in geometric group theory), we instead are interested in an intermediate range of scale parameters $r$.
Indeed, if $r$ is smaller than the distance between any two data points in $X$, then $\vr{X}{r}=X$ is a disjoint union of points.
Conversely, if $r$ is larger than the diameter of $X$, then $\vr{X}{r}$ is necessarily contractible.
Neither of these regimes help us describe the ``shape'' of dataset $X$.
Instead, the interesting topology appears when scale $r$ is varied in an intermediate regime, as computed by persistent homology.
These varying regimes of scale parameters ($r$ small, $r$ intermediate, $r$ large) are analogous to the subcritical, critical, and super-critical regimes in random topology~\cite{bobrowski2018topology,kahle2011random}.

As a finite dataset $X$ converges (say, as more samples are drawn) to an underlying infinite space $M$, the persistent homology of $\vr{X}{r}$ converges to that of $\vr{M}{r}$~\cite{ChazalDeSilvaOudot2013}. 
There has thus been interest in the literature to identify the homotopy types of the Vietoris--Rips complexes of manifolds.
Essentially, the only examples that are fully understood are the Vietoris--Rips complexes of the circle~\cite{AA-VRS1,AAFPP-J}, which obtain the homotopy types of all odd spheres as the scale parameter increases, before they finally become contractible.
We have a countably infinite number of ``phase transitions'' from one odd-dimensional sphere $S^{2k-1}$ to the next one $S^{2k+1}$ as the scale increases, demonstrating the complexity of the situation.
Vietoris--Rips thickenings of $n$-spheres for $n>1$ are understood only up to the first change in homotopy type~\cite{AAF}.
The 1-dimensional persistent homology of geodesic spaces is also understood~\cite{gasparovic2018complete,virk20181}.

In this paper we take one step towards merging the perspectives on Vietoris--Rips complexes provided by geometric group theory and by applied topology.
We study Vietoris--Rips complexes of spaces which are equipped with a group action (as in geometric group theory) but in the range of intermediate scale parameters (as in applied topology). 
More specifically, let $G$ be a group acting properly and by isometries on a metric space $X$; it follows that the quotient space $X/G$ is a metric space.
We study the Vietoris--Rips complexes of the quotient space $X/G$.
Our first results are for small scale parameters (but not tending to zero), in which we are able to show that the Vietoris--Rips complex of the quotient, namely $\vr{X/G}{r}$, is isomorphic to the quotient of the Vietoris--Rips complex, namely $\vr{X}{r}/G$.
We furthermore identify which scale parameters lie in this regime in terms of the quantitative properties of the group action.
Our results apply not only for the Vietoris--Rips simplicial complex but also for the Vietoris--Rips metric thickening~\cite{AAF}, and we give analogous results for \v{C}ech simplicial complexes and \v{C}ech metric thickenings.

We also consider a slightly larger regime of scale parameters for projective spaces.
Let $S^n$ be the $n$-sphere equipped with the geodesic metric\footnote{Analogous results also hold with the Euclidean metric on $S^n$, with the relevant scale parameters being adjusted accordingly, and with no change to the homotopty types.
We restrict attention to the geodesic metric for convenience.} such that the circumference of any great circle is one.
The sphere is naturally equipped with a $G=(\lbrace \pm 1\rbrace, \times)\cong\Z/2\Z$ action, which exchanges each point $x$ with its antipode $-x$.
Let $\RP^n=S^n/(x\sim-x)=S^n/G$ be real projective space equipped with the quotient metric.
Note that with the quotient metric, the circumference of any great circle in $\RP^n$ is $\frac{1}{2}$.
We demonstrate that $\vr{\RP^n}{r}$ is homotopy equivalent to $\RP^n$ for all $r$ less than $\tfrac{1}{6}$, which is the diameter of an inscribed equilateral triangle in any great circle of $\RP^n$.
Furthermore, we study the metric thickening $\vrmleq{\RP^n}{r}$ at the first scale parameter, namely $r=\frac{1}{6}$, where the homotopy type changes.
In doing so, we leverage the fact that $\RP^n$ is the quotient of $S^n$ under the antipodal action.
We prove that $\vrmleq{\RP^n}{\frac{1}{6}}$ has the homotopy type of a $(2n+1)$-dimensional CW complex and hence has trivial homology and cohomology in dimensions $2n+2$ and above.

As one example application of our work, suppose $X$ is an unknown space of confirmations of a molecule, or perhaps only those confirmations of a molecule whose associated energy is bounded from above by a chosen energy cutoff.
If the molecule has a group $G$ of symmetries, then $G$ will act on the space $X$.
Given a random sample $Y_n$ of $n$ points from $X$, recovered for example by molecular dynamics, one might try to estimate the topology of conformation space $X$ by computing the persistent homology of the Vietoris--Rips complexes $\vr{Y_n}{r}$ as $r$ varies.
As one forms a denser and denser sample by increasing the number of random points $n$, the persistent homology of $\vr{Y_n}{r}$ converges to that of $\vr{X}{r}$ and hence can be used to estimate the homology groups of $X$~\cite{ChazalDeSilvaOudot2013}.
How would this experiment compare if instead one first quotiented out by the molecular symmetries $G$ and instead considered a finite sample $Y'_n$ of $n$ points from $X/G$?
As $n$ goes to infinity, the persistent homology of $\vr{Y'_n}{r}$ will converge to that of $\vr{X/G}{r}$.
Our results show that these two experiments are consistent in the following sense.
For scale $r$ small enough, the quotient of $\vr{X}{r}$ by the symmetry group $G$ is isomorphic to $\vr{X/G}{r}$ as simplicial complexes, and therefore quotienting out by the group of symmetries affects the experiment, and the predicted topological types, in a way that is understood.
Furthermore, we give precise bounds on which scale parameters $r$ are small enough for such results to hold.

\section{Preliminaries}

We recall a few standard preliminaries in point set topology and algebraic topology that will lead into an introduction to Vietoris--Rips and \v{C}ech simplicial complexes and thickenings.

\subsection*{Metric spaces}

A metric space $(X, d)$ is a set $X$ equipped with a metric $d \colon X\times X\to \R$ satisfying the following properties:

\begin{itemize}
\item $d(x,y)$ is a nonnegative real number for all choices of $x$ and $y$ in $X$,
\item $d(x,y)$ is zero if and only if $x = y$, and
\item $d(x,y) + d(y,z) \geq d(x, z)$.
\end{itemize}
For $x\in X$ and $r>0$, we let $B(x,r)=\{y\in X~|~d(x,y)<r\}$ denote the open ball in $X$ of radius $r$ about $x$.
Given a subset $Y\subseteq X$ of a metric space, we let $\diam(Y)=\sup\{d(x,x')~|~x,x'\in Y\}$ denote the diameter of this subset.
Metric spaces are a commonly studied topic in mathematics, and they are generalized by topological spaces, which also have a notion of open neighborhoods but need not have a notion of distance.




\subsection*{Simplicial complexes}

A simplex on the vertices $v_0, v_1, v_2, \ldots, v_k$ may be thought of as the convex hull of these points when they are placed at the location of the standard basis vectors $e_i$ in Euclidean space.
A \emph{simplicial complex} is a union of simplices joined together by gluing maps.
More precisely, given a set of vertices $V$, an \emph{abstract simplicial complex} $K$ is a collection of subsets of $V$ (called \emph{simplices}) containing all singleton sets, with the property that if $\sigma\in K$ is a simplex and $\tau\subseteq \sigma$, then we also have $\tau\in K$.
The geometric realization of a simplicial complex is a way to turn this combinatorial data into a topological space containing vertices, edges, triangles, tetrahedra, and so forth; in this paper we identify abstract simplicial complexes with their geometric realizations.

\subsection*{Vietoris--Rips simplicial complexes}

Let $X$ be a metric space and let $r\ge 0$ be a scale parameter.
A \emph{Vietoris--Rips simplicial complex} $\vrleq{X}{r}$ is a simplicial complex with vertex set $X$ in which the simplex $\{x_0, x_1, \ldots, x_k\}$ is in the complex if, for all $0\le i,j\le k$, the pairwise distance between $x_i$ and $x_j$ is at most $r$.
We instead write $\vrless{X}{r}$ when the pairwise distances are required to be strictly less than $r$ and $\vr{X}{r}$ when the distinction is not important.

\subsection*{\v{C}ech simplicial complexes}

Let $X$ be a metric space and let $r\ge 0$ be a scale parameter.
A \emph{\v{C}ech simplicial complex} $\cechless{X}{r}$ is a simplicial complex with vertex set $X$ in which the simplex $\{x_0, x_1, \ldots, x_k\}$ is in the complex if $\cap_{i=0}^k B(x_i,r)\neq\emptyset$.
We write $\cechleq{X}{r}$ to instead specify the use of closed balls $B_\le(x,r)=\{y\in X~|~d(x,y)\le r\}$.

\subsection*{Vietoris--Rips and \v{C}ech metric thickenings}

As sets, the Vietoris--Rips and \v{C}ech metric thickenings~\cite{AAF} are identical to the geometric realizations of the corresponding Vietoris--Rips and \v{C}ech simplicial complexes.
However, they are equipped with a different topology, indeed a metric, which sometimes produces a different (and often more natural) homeomorphism type, or it can even produce a different homotopy type.

More explicitly, let $X$ be a metric space and let $r\ge 0$.
The \emph{Vietoris--Rips metric thickening} is the set
\[ \vrmleq{X}{r}=\left\{\sum_{i=0}^k\lambda_i\delta_{x_i}~|~k\in\N,\ \lambda_i\ge0,\ \sum\lambda_i=1,\ \diam(\{x_0,\ldots,x_k\})\le r\right\}, \]
equipped with the 1-Wasserstein metric~\cite{edwards2011kantorovich,kellerer1984duality,kellerer1982duality}.
Here $\delta_x$ denotes the Dirac delta probability measure with mass one at $x\in X$.
Roughly speaking, one can think of a measure as a mass distribution.
From this viewpoint, two measures can be thought of as two mass distributions, and the 1-Wasserstein distance between the two measures is the minimum amount of work required to transport the mass in the first mass distribution to the mass in the second mass distribution.
This is sometimes called the earth mover's distance~\cite{villani2008optimal}.
The definition for $\vrmless{X}{r}$ is analogous.

We remark that the inclusion $X\hookrightarrow\vrm{X}{r}$ defined by $x\mapsto \delta_x$ is continuous (in fact an isometry onto its image), whereas the analogous inclusion $X\hookrightarrow\vr{X}{r}$ into the simplicial complex is not continuous for $X$ not discrete.
It is also worth remarking that the simplicial complex $\vr{X}{r}$ is not metrizable if it is not locally finite by Proposition~4.2.16(2) of~\cite{sakai2013geometric} even though the input $X$ is a metric space.
By contrast, the thickening $\vrm{X}{r}$ is always a metric space, and furthermore $\vrm{X}{r}$ is an \emph{$r$-thickening} of $X$ by Lemma~3.6 of~\cite{AAF}.

The \emph{\v{C}ech metric thickening} is the set
\[ \cechmless{X}{r}=\left\{\sum_{i=0}^k\lambda_i\delta_{x_i}~|~k\in\N,\ \lambda_i\ge0,\ \sum\lambda_i=1,\ \cap_{i=0}^k B(x_i,r)\neq\emptyset\right\}, \]
again equipped with the 1-Wasserstein metric.

For the remainder of the paper, we refer to a point $\sum \lambda_i \delta_{x_i}$ in a metric thickening simply as $\sum \lambda_i x_i$.
This allows us to let $\sum \lambda_i x_i$ refer to either a point in a metric thickening or to a point in the geometric realization of a simplicial complex.

\section{Group actions and Vietoris--Rips thickenings}\label{sec:group-actions}

Let a group $G$ act on a metric space $X$.
We study Vietoris--Rips complexes and thickenings of $X/G$.
We begin in Section~\ref{ssec:metrizable} by describing when $X/G$ is metrizable.
In Section~\ref{ssec:quantitative-actions}, we survey additional metric assumptions on the action of $G$ on $X$.
These additional assumptions allow us, in Section~\ref{ssec:thickenings}, to relate $\vr{X/G}{r}$ to $\vr{X}{r}/G$.
We end with examples in Section~\ref{ssec:examples}.

\subsection{Metrizable quotient spaces}\label{ssec:metrizable}

An action of a group $G$ on a set $X$ is a function $G \times X \to X$, denoted $(g,x)\mapsto g\cdot x$, satisfying $g \cdot (h \cdot x) = (gh) \cdot x$ for all $g,h\in G$ and $x\in X$, and satisfying $e\cdot x = x$ for all $x\in X$ (where $e$ is the identity element of $G$).
The orbit of an element $x\in X$, under the action of $G$, is the set $\mathcal O(x)=\lbrace g \cdot x ~|~ g \in G \rbrace$. Note that $x\in \mathcal O(x)$ and that, for any two elements $x,y\in X$, either $\mathcal O(x)=\mathcal O(y)$ or $\mathcal O(x)\cap \mathcal O(y)=\emptyset$.
As a consequence, the orbits of the group action partition $X$.

Let $G$ be a group acting on a metric space $X$.
We say that $G$ \emph{acts by isometries} on $X$ if, for each $g\in G$, the map $g\colon X\to X$ defined by $x\mapsto g\cdot x$ is an isometry.
In other words, we have a homomorphism from $G$ into the group of isometries of $X$.
Furthermore, we say that the action of $G$ on $X$ is \emph{proper} if, for each $x\in X$, there exists some $r>0$ such that the set $\{g\in G~|~g\cdot B(x,r)\cap B(x,r)\neq\emptyset\}$ is finite.
In particular, an action by a finite group is necessarily proper.

For $x\in X$, we let $[x]$ denote the corresponding orbit in $X/G$.
It follows from Proposition~I.8.5 of~\cite{bridson2011metric}
that if $G$ acts properly by isometries on $X$, then the quotient space $X/G$ is itself a metric space.
Its \emph{quotient metric} is defined via
\begin{equation}\label{eq:quotient-metric}
d_{X/G}([x],[x'])=\inf_{g\in G}d_X(x,g\cdot x').
\end{equation}
The assumption of a proper action rules out examples such as the action of the rationals $\Q$ on the reals $\R$ by addition in which the quotient space $\R/\Q$ is not metrizable.\footnote{
Other contexts in which $X/G$ is a metric space, with the quotient metric as described above, are in Section~5 of~\cite{bridson2011metric},
Sections~3 and 10 of~\cite{BuragoBuragoIvanov}, and
Chapters~4--7 of~\cite{herman1968quotients}.
}

If the metric space $X$ is equipped with an isometric action of $G$, it follows that the Vietoris--Rips complexes $\vr{X}{r}$ are also equipped with an action of $G$.
Indeed, given any point $\sum \lambda_i x_i \in \vr{X}{r}$ and $g\in G$, we define $g \cdot \sum \lambda_i x_i = \sum \lambda_i g\cdot  x_i$.
For $\sum \lambda_i x_i\in\vr{X}{r}$, we let $[\sum \lambda_i x_i]$ denote the corresponding orbit in $\vr{X}{r}/G$.
Analogous actions can be defined on the Vietoris--Rips thickening $\vrm{X}{r}$ as well as on \v{C}ech complexes and thickenings.

If $G$ acts properly by isometries on $X$, then $X/G$ is a metric space with distance given by~\eqref{eq:quotient-metric}, and so we can define its Vietoris--Rips and \v{C}ech simplicial complexes and thickenings.
Our goal in this section will be to explain the relationship between $\vr{X/G}{r}$ and $\vr{X}{r}/G$ when $r$ is small, and analogously for the \v{C}ech and the metric thickening versions.

\subsection{Different types of group actions}\label{ssec:quantitative-actions}

We now survey a list of increasingly stringent properties that the action of $G$ on $X$ could satisfy.
The definition of a \emph{free action} requires $X$ to be a set, the definition of a \emph{covering space action} requires $X$ to be a topological space, and the definitions of an \emph{$r$-ball action} and an \emph{$r$-distance action} require $X$ to be a metric space.
Free actions and covering space actions are classical: see~\cite{bridson2011metric} for a wide variety of properties that a group acting on a metric space could satisfy, such as being faithful, free, cocompact, or proper.
We introduce $r$-ball and $r$-distance action as quantitative versions of these properties.
\begin{itemize}
\item The action of $G$ on $X$ is \emph{free} if $g\cdot x=x$ for any $x\in X$ implies that $g$ is the identity element in $G$.
\item The action of $G$ on $X$ is a \emph{covering space action} if every point $x\in X$ has a neighborhood $U\ni x$ such that if $U\cap g\cdot U\neq\emptyset$, then $g$ is the identity element in $G$.
See Section~1.3 of~\cite{Hatcher} or pages 311--312 of~\cite{lee2010introduction} for more details; the term \emph{covering space action} is introduced by Hatcher in part in order to disambiguate terminology.
\item For $r>0$, we define the action of $G$ on $X$ to be an \emph{$r$-ball action} if $B(x,r)\cap g\cdot B(x,r)\neq\emptyset$ for any $x\in X$ implies that $g$ is the identity element in $G$.
\item For $r>0$, we define the action of $G$ on $X$ to be an \emph{$r$-distance action} if $d_X(x,g\cdot x)<r$ for any $x\in X$ implies that $g$ is the identity element in $G$.
\end{itemize}
We have the following sequence of proper inclusions for $r>0$:
\[2r\text{-distance actions}\subset r\text{-ball actions}\subset\text{covering space actions}\subset\text{free actions}.\]

\begin{definition}\label{def:r-diameter}
Let $r>0$.
The action of $G$ on a metric space $X$ is an \emph{$r$-diameter action} when, for any nonnegative integer $k$, $\diam_{X/G}\{[x_0],\ldots,[x_k]\}< r$ implies that there exists a unique choice of elements $g_i\in G$ for $1\le i\le k$ such that $\diam_X\{x_0,g_1\cdot x_1\ldots,g_k\cdot x_k\}=\diam_{X/G}\{[x_0],\ldots,[x_k]\}$.
\end{definition}

We claim that an $r$-diameter action is also an $r$-distance action.
Indeed, suppose $r>d_X(x,g\cdot x)\ge\diam_{X/G}\{[x],[g\cdot x]\}$.
Then the $r$-diameter action assumption implies there exists a unique element $g\in G$ such that $d_X(x,g\cdot x)=\diam_X\{x,g\cdot x\}=\diam_{X/G}\{[x],[g\cdot x]\}=0$, i.e.\ necessarily $x=g\cdot x$, and so $g$ is the identity since $r$-diameter actions are free.
Hence, this is an $r$-distance action.

However, we give the following example to show that $r$-distance actions are not necessarily $r$-diameter actions.

\begin{example}
\label{ex:ZactR}
Let $G=\Z$ act on $X=\R$ by translation, i.e., for any $g\in \Z$ and $x\in \R$ we have $g\cdot x=g + x$.
Note this action is a $1$-distance action,
since if $1>d(x,g\cdot x)=d(x,g+x)=|g|$, then $g=0$ is the identity in $\Z$.
Clearly, it is not a $(1+\varepsilon)$-distance action for any $\varepsilon>0$.
This action is not an $r$-diameter action for any $r>\frac{1}{3}$ since we have $\diam_{\R/\Z}\{[0],[\frac{1}{3}],[\frac{2}{3}]\}=\frac{1}{3}$, but $\min_{g_1,g_2\in\Z}\diam_\R\{0,g_1+\frac{1}{3},g_2+\frac{2}{3}\}=\frac{2}{3}$.
One can check that this is an $r$-diameter action for $r\leq\frac{1}{3}$.
\end{example}

\begin{definition}\label{def:r-nerve}
Let $r>0$.
The action of $G$ on $X$ is an \emph{$r$-nerve action} when, for any nonnegative integer $k$, $\cap_{i=0}^k B_{X/G}([x_i],r)\neq \emptyset$ implies that there exists a unique choice of elements $g_i\in G$ for $1\le i\le k$ such that $\cap_{i=0}^k B_{X}(g_i\cdot x_i,r)\neq \emptyset$, where we require $g_0$ is the identity element of $G$.
\end{definition}

We now demonstrate that $\Z$ acting on $\R$ by translation is both a $\frac{1}{2}$-ball action and a $\frac{1}{4}$-nerve action, with the $r$ parameters being as high as possible, demonstrating that $r$-nerve actions and $r$-ball actions are distinct concepts.

\begin{example}\label{ex:Z-act-R}
We first show that $\Z$ acting on $\R$ by translation is a $\frac{1}{2}$-ball action.
Select an arbitrary open ball of radius $\frac{1}{2}$ in $\R$.
Then, it follows that this action is a $\frac{1}{2}$-ball action since, for any $0\neq g \in \Z$, the balls $B(x, \frac{1}{2})$ and $g\cdot B(x,\frac{1}{2})$ do not intersect.
So $\Z$ acting on $\R$ by translation is a $\frac{1}{2}$-ball action.

We will now demonstrate that $\Z$ acting on $\R$ by translation is not a $(\frac{1}{2} + \varepsilon)$-ball action for any $\varepsilon > 0$.
Now, take an arbitrary $\frac{1}{2} + \varepsilon$ ball.
Then, we note that $B(x, \frac{1}{2} + \varepsilon)$ and the action of $g = 1 \in \Z$ result in an intersection of $B(x, r)$ and $g\cdot B(x, r)$ and yet $g$ is not the identity element.
Therefore, $\Z$ acting on $\R$ by translation is not a $(\frac{1}{2} + \varepsilon)$-ball action for any $\varepsilon > 0$.

Now, we show that $\Z$ acting on $\R$ is a $\frac{1}{4}$-nerve action.
Take an arbitrary set of intersecting balls of radius $\frac{1}{4}$, namely $\cap_{i=0}^k B_{\R/\Z}([x_i],\frac{1}{4})\neq\emptyset$.
These balls are open intervals of length half the circumference of the circle, and, hence, they intersect in a single, smaller, connected interval.
It follows that once $g_0=0\in \Z$ is chosen, there exists a unique choice of elements $g_i\in \Z$ for $1\le i\le k$ such that $\cap_{i=0}^k B_{\R}(g_i\cdot x_i,r)\neq \emptyset$, as required (see Figure~\ref{fig:Z-act-R}).

Finally, for any $\varepsilon > 0$, the action of $\Z$ on $\R$ by translation is not a $(\frac{1}{4} + \varepsilon)$-nerve action, demonstrating that $\frac{1}{4}$ is indeed the maximum parameter we can use.
To see this, take two balls of radius $\frac{1}{4} + \varepsilon$ in $\R/\Z$ centered at $[0]$ and $[\frac{1}{2}]$.
These two balls intersect at both $[\frac{1}{4}]$ and $[\frac{3}{4}]$ in $\R/\Z=S^1$.
The two components of intersection correspond to the fact that when we lift to balls of the form $B_\R(0,\frac{1}{4} + \varepsilon)$ and $g_1\cdot B_\R(\frac{1}{2},\frac{1}{4} + \varepsilon)$,
we can maintain a nontrivial intersection either by choosing $g_1=0$ or $g_1=-1$.
Since this choice of $g_1$ is not unique, the action of $\Z$ on $\R$ by translation is not a $(\frac{1}{4} + \varepsilon)$-nerve action.
\end{example}

\begin{figure}[htb]
\centering
\includegraphics[width=0.55\textwidth]{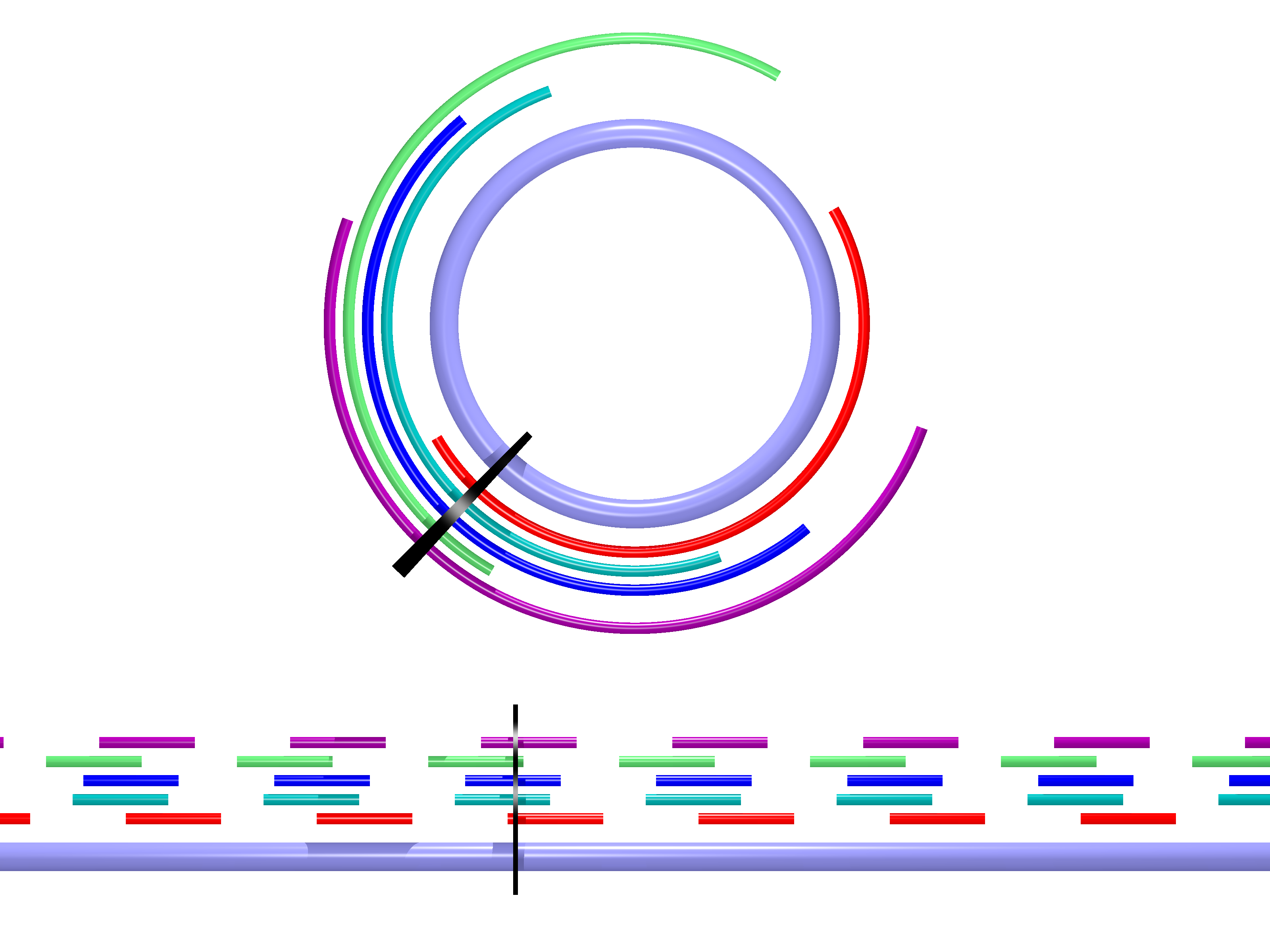}
\caption{As explained in Example~\ref{ex:Z-act-R}, the action of $\Z$ on $\R$ is a $\frac{1}{4}$-nerve action.
Indeed, let $x_0\in \R$, indicated above by the black tick mark, and consider the ball $B_{\R/\Z}([x_0],\frac{1}{4})$, drawn as the blue arc on the circle.
For any other open ball of radius $\frac{1}{4}$ in the circle, there is at most one lift of this ball to $\R$ that intersects $B_{\R}(x_0,\frac{1}{4})$.
}
\label{fig:Z-act-R}
\end{figure}


\subsection{Group actions and metric thickenings}\label{ssec:thickenings}

Let $G$ be a group acting properly and by isometries on a metric space $X$.
We now explain how to understand the Vietoris--Rips and \v{C}ech simplicial complexes and metric thickenings of $X/G$ for sufficiently small scale parameters depending on the behavior of the action of $G$.

\begin{proposition}\label{prop:group-action}
Let $G$ be a group acting properly and by isometries on a metric space $X$.
If the action is a $t$-diameter action, then 
\begin{itemize}
    \item $\vrless{X/G}{r}$ is isomorphic to $\vrless{X}{r}/G$ for all $r\le t$,
    \item $\vrleq{X/G}{r}$ is isomorphic to $\vrleq{X}{r}/G$ for all $r<t$,
    \item $\vrmless{X/G}{r}$ is homeomorphic to $\vrmless{X}{r}/G$ for all $r\le t$, and
    \item $\vrmleq{X/G}{r}$ is homeomorphic to $\vrmleq{X}{r}/G$ for all $r<t$.
\end{itemize}
\end{proposition}

\begin{proof}
We first consider the case of the Vietoris--Rips simplicial complexes.
We can handle the first two bullet points simultaneously simply because a simplex in either complex has diameter less than $t$.

Consider the simplicial map $h\colon\vr{X}{r}\to\vr{X/G}{r}$ defined by $h(x)=[x]$; on geometric realizations this is defined via $h(\sum \lambda_i x_i)=\sum\lambda_i[x_i]$.
This map is well-defined since $G$ acts isometrically.
Note that if two points in the geometric realization of $\vr{X}{r}$ are in the same orbit of the $G$ action on the geometric realization, then they have the same image under $h$.
It follows that $h$ induces a map $\tilde{h}\colon\vr{X}{r}/G\to\vr{X/G}{r}$.
We will show that $\tilde{h}$ is a homeomorphism.

We need to show the following two facts.
\begin{enumerate}
\item Map $\tilde{h}$ is surjective.
\item Map $\tilde{h}$ is injective.
\end{enumerate}

For (1), note that $\tilde{h}$ is surjective if $h$ is surjective.
The map $h$ is surjective because, given any simplex $\sigma=\{[x_0],\ldots,[x_k]\}\in\vr{X/G}{r}$, by the definition of an $r$-diameter action, there exists a simplex $\sigma'=\{x_0,g_1\cdot x_1\ldots,g_k\cdot x_k\}\in\vr{X}{r}$ with $h(\sigma')=\sigma$.

For (2), we would like to consider any two points $[\sum \lambda_i x_i],[\sum \lambda'_j x'_j]\in \vr{X}{r}/G$ with $\tilde{h}([\sum \lambda_i x_i]) = \tilde{h}([\sum \lambda'_j x'_j])$.
This means that $h(\sum \lambda_i x_i)=h(\sum \lambda'_j x'_j)$, i.e., that $\sum\lambda_i[x_i]=\sum\lambda'_j[x'_j]$.
It suffices to show that there is some $g\in G$ with $g\cdot \sum \lambda_i x_i = \sum \lambda'_j x'_j$.
This follows from the ``uniqueness'' part of the definition of an $r$-diameter action.
Indeed, given any simplex $\sigma=\{[x_0],\ldots,[x_k]\}\in\vr{X/G}{r}$, there exists a unique simplex $\tilde{\sigma}=\{x_0,g_1\cdot x_1\ldots,g_k\cdot x_k\}\in\vr{X}{r}$ containing $x_0$ with $h(\sigma')=\sigma$ and hence a unique simplex $\sigma''\in\vr{X}{r}/G$ with $\tilde{h}(\sigma'')=\sigma$.

For the case of Vietoris--Rips metric thickenings, we consider the analogous map $h\colon\vrm{X}{r}\to\vrm{X/G}{r}$ defined by $h(\sum \lambda_i x_i)=\sum\lambda_i[x_i]$; this map is well-defined since $G$ acts isometrically.
The only additional observation to make in this case is that both $h$ and its inverse are continuous.
\end{proof}

The \v{C}ech case is similar.

\begin{proposition}\label{prop:group-action-cech}
Let $G$ be a group acting properly and by isometries on a metric space $X$.
If the action is a $t$-nerve action, then 
\begin{itemize}
    \item $\cechless{X/G}{r}$ is isomorphic to $\cechless{X}{r}/G$ for all $r\le t$,
    \item $\cechleq{X/G}{r}$ is isomorphic to $\cechleq{X}{r}/G$ for all $r<t$,
    \item $\cechmless{X/G}{r}$ is homeomorphic to $\cechmless{X}{r}/G$ for all $r\le t$, and
    \item $\cechmleq{X/G}{r}$ is homeomorphic to $\cechmleq{X}{r}/G$ for all $r<t$.
\end{itemize}
\end{proposition}

\begin{proof}
We first consider the case of the \v{C}ech simplicial complexes.
Consider the simplicial map $h\colon\cech{X}{r}\to\cech{X/G}{r}$ defined by $h(x)=[x]$; on geometric realizations this is defined via $h(\sum \lambda_i x_i)=\sum\lambda_i[x_i]$.
This map is well-defined since $G$ acts isometrically.
Note that if two points in the geometric realization of $\cech{X}{r}$ are in the same orbit of the $G$ action, then they have the same image under $h$.
It follows that $h$ induces a map $\tilde{h}\colon\cech{X}{r}/G\to\cech{X/G}{r}$.
We will show that $\tilde{h}$ is a homeomorphism.

Again, we need to show the following two facts.
\begin{enumerate}
\item Map $\tilde{h}$ is surjective.
\item Map $\tilde{h}$ is injective.
\end{enumerate}

For (1), note that $\tilde{h}$ is surjective if $h$ is surjective.
The map $h$ is surjective because, given any simplex $\sigma=\{[x_0],\ldots,[x_k]\}\in\cech{X/G}{r}$, by the definition of an $r$-nerve action, there exists a simplex $\sigma'=\{x_0,g_1\cdot x_1\ldots,g_k\cdot x_k\}\in\cech{X}{r}$ with $h(\sigma')=\sigma$.

For (2), we would like to consider any two points $[\sum \lambda_i x_i],[\sum \lambda'_j x'_j]\in \cech{X}{r}/G$ with $\tilde{h}([\sum \lambda_i x_i]) = \tilde{h}([\sum \lambda'_j x'_j])$.
This means that $h(\sum \lambda_i x_i)=h(\sum \lambda'_j x'_j)$, i.e., that $\sum\lambda_i[x_i]=\sum\lambda'_j[x'_j]$.
It suffices to show that there is some $g\in G$ with $g\cdot \sum \lambda_i x_i = \sum \lambda'_j x'_j$.
This follows from the ``uniqueness'' part of the definition of an $r$-nerve action.
Indeed, given any simplex $\sigma=\{[x_0],\ldots,[x_k]\}\in\cech{X/G}{r}$, there exists a unique simplex $\tilde{\sigma}=\{x_0,g_1\cdot x_1\ldots,g_k\cdot x_k\}\in\cech{X}{r}$ containing $x_0$ with $h(\sigma')=\sigma$ and hence a unique simplex $\sigma''\in\cech{X}{r}/G$ with $\tilde{h}(\sigma'')=\sigma$.

For the case of \v{C}ech metric thickenings, it suffices to observe that the map $h$ and its inverse are also continuous as maps on the \v{C}ech thickenings.
\end{proof}

\subsection{Examples}\label{ssec:examples}

We now look at some examples of groups acting isometrically on topological spaces.
Although we don't often know the precise homotopy types of Vietoris--Rips complexes of arbitrary spaces, we can sometimes address the relationship between $\vr{X/G}{r}$ and the quotient of $\vr{X}{r}$ under the group action.

\begin{example}
\label{ex:VR-ZactR}
Let $G=\Z$ act on $X=\R$ by translation.
From Example~\ref{ex:ZactR} we know this is an $r$-diameter action for $r\le\frac{1}{3}$.
Hence if $S^1=\R/\Z$ is the circle of unit circumference, then Proposition~\ref{prop:group-action} implies that for $r<\frac{1}{3}$, the complex $\vr{S^1}{r}=\vr{\R/\Z}{r}$ is isomorphic to $\vr{\R}{r}/\Z$, which is homotopy equivalent to $S^1$.
By contrast, the action of $\Z$ on $\R$ is not an $r$-diameter action for any $r>\frac{1}{3}$, and also $\vr{S^1}{r}$ is not homotopy equivalent to $S^1$ for any $r>\frac{1}{3}$ by~\cite{AA-VRS1}.
Hence this example shows the tightness of the bounds on $r$ in Proposition~\ref{prop:group-action}.
\end{example}

\begin{example}\label{ex:12-circles}
Suppose $X$ consists of several connected components which are all isometric.
Let $G$ act by isometries on $X$, and suppose furthermore that $G$ acts freely on the connected components of $X$, meaning that if $x, g\cdot x \in X$ are in the same connected component, then $g$ is the identity element.
Suppose no two distinct connected components are within distance $t$ of each other.
It follows that $G$ is a $t$-diameter action on $X$.
Therefore, Proposition~\ref{prop:group-action} implies that for $r<t$, the complex $\vr{X}{r}$ is isomorphic to the disjoint union $\coprod^{|G|}\vr{X/G}{r}$.

For example, let $Y$ be a circle in $\R^3$ with center at $(\frac{5}{8},\frac{3}{8},-\frac{\sqrt{2}}{{8}})$, lying in the plane with normal vector $(1,1,0)$, and with radius $\frac{1}{5}$
using the Euclidean metric.
The action of $G = A_4$, as the group of rotational symmetries of a particular regular tetrahedron centered about the origin (with vertex coordinates $(\pm \frac{1}{2}, 0, -\frac{\sqrt{2}}{4})$, $(0, \pm \frac{1}{2}, \frac{\sqrt{2}}{4})$), extends to $\R^3$.
The orbit of the circle, $Y$, consists of $12$ copies of $Y$; see Figure~\ref{fig:12-circles}.
We let $X$ denote the union of these $12$ copies, so $X/G=Y$.
We obtain that $\vr{X}{r}$ is isomorphic to the disjoint union of $12$ copies of $\vr{X/G}{r}$ for all $r$ smaller than the distance between connected components in $X$.
\end{example}

\begin{figure}[htb]
\centering
\includegraphics[width=0.55\textwidth]{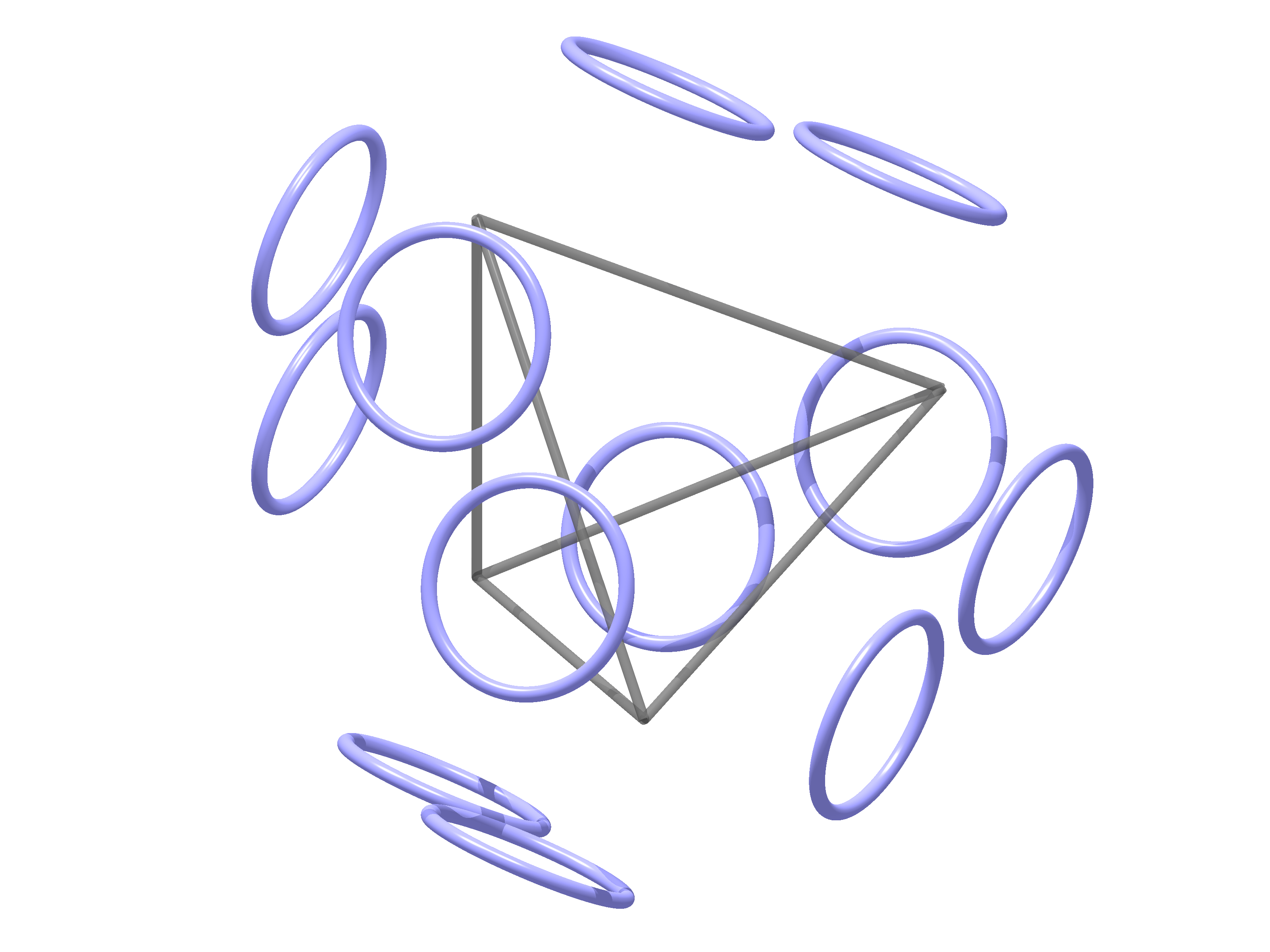}
\caption{In Example~\ref{ex:12-circles}, the space $X$ is a set of $12$ circles in $\R^3$.}
\label{fig:12-circles}
\end{figure}

In some cases we can conjecture the full homotopy type over all scale parameters $r$, as the next example shows.

\begin{example}
Take the unit circle with center $(4, 0)$ in $\R^2$.
If we rotate this circle about the origin under the action of $G = \Z / 6\Z$, i.e.\ by rotating $\R^2$ about the origin in multiples of $60$ degrees, we obtain $6$ unit circles (with centers $(\pm 4,0)$, $(\pm 2, 2\sqrt{3})$, $(\pm 2, -2\sqrt{3})$).
We let $X$ denote the union of the $6$ circles.
The closest distance between two adjacent circles in $X$ is 2.
By Example~\ref{ex:12-circles} and Proposition~\ref{prop:group-action}, for $r<2$ we have an isomorphism $\vr{X}{r}\cong\coprod^{6}\vr{X/G}{r}$. 
The homotopy types of the Vietoris--Rips complexes of the unit circle $\vr{X/G}{r}$ are known for all $r$~\cite{AA-VRS1}.
We obtain $\vrless{X}{r}\cong\coprod^{6}S^{2k+1}$ for all $0<r<2$, where the integer $k$ is monotonically nondecreasing with $r$.

For larger scale parameters $r>2$, we can form conjectures by noting that the Vietoris--Rips complex of each individual circle is contractible, but that we have six unit circles that are evenly-spaced around a larger circle of radius 4.
Think of each of the six circles, momentarily, as a single point, giving six evenly-spaced points.
Vietoris--Rips complexes of evenly-spaced points on the circle have been studied in~\cite{Adamaszek2013,AAFPP-J}.
In particular, the Vietoris--Rips complex of six-evenly spaced points on the circle, as the scale increases, obtains the homotopy types of six disjoint points, the circle $S^1$, the two-sphere $S^2$, and finally the contractible space.
This knowledge allows us to conjecture the homotopy type of $\vr{X}{r}$ at $r > 2$ when the six circles join up but individually are contractible.
Indeed, we conjecture that the successive homotopy types of $\vrless{X}{r}$ are the following, starting from $r=0$ and going up to the diameter of the whole space: $\coprod^6 S^1$, $\coprod^6 S^3$, $\coprod^6 S^5$, $\coprod^6 S^7$, $\coprod^6 S^9$, \ldots for $0<r<2$ (this part is proven), $S^1$ for $2<r<4\sqrt{3}-2$, $S^2$ for $4\sqrt{3}-2<r<6$, and finally the contractible space for $r>6$.
\end{example}

\begin{example}
Take the torus $X$ with the flat metric (i.e., $X$ is a quotient of $\R^2$) under the action of $G\cong\Z/14\Z\cong\Z/2\Z\times\Z/7\Z$, defined as follows.
Take the torus to be $[0,2\pi]\times[0,2\pi]$ with the top and bottom (respectively, left and right) edges identified.
Then, identify the points $[x, y]$ with $[x, y + \frac{2\pi}{7}]$ and $[x + \pi, y]$.
The quotient space is a torus with different distances for traveling along geodesics around the short loop and long loop in this torus.

The action of $G$ is a $\frac{2\pi}{21}$-diameter action, but not a $(\frac{2\pi}{21} + \varepsilon)$-diameter action for any $\varepsilon>0$.
If we take the points $[0, 0]$, $[0, \frac{2\pi}{21}]$, and $[0, \frac{4\pi}{21}]$, then we note that the diameter of these points in the quotient metric is $\frac{2\pi}{21}$, but there is no choice of elements in $G$ so that the diameter of the corresponding lifted points in $X$ is $\frac{2\pi}{21}$.
However, if an arbitrary number of points are selected next to an arbitrary point $[x_0]\in X/G$ that are of diameter less than $\frac{2\pi}{21}$, then there exists a unique choice of elements in $G$ so that $\diam_X\{x_0,g_1\cdot x_1\ldots,g_k\cdot x_k\}=\diam_{X/G}\{[x_0],\ldots,[x_k]\}$.
This follows since, when moving along a geodesic in $X/G$, it requires at least $\frac{2\pi}{7}$ in path length to return back to the initial starting point.

We deduce from Proposition~\ref{prop:group-action} that $\vr{X/G}{r} = \vr{X}{r}/G$ for $r < \frac{2\pi}{21}$, though we do not know what the homotopy types of these Vietoris--Rips complexes of tori are.\footnote{Homotopy types of Vietoris--Rips complexes of tori with the $L^\infty$ or supremum metric are fully understood by Proposition~10.2 of~\cite{AA-VRS1}.}

A related group action on the torus $X$, i.e.\ the square $[0,2\pi]\times[0,2\pi]$ with sides identified, is by the dihedral group $D_7$ with fourteen elements.
Put 7 equally spaced points of the form $(\frac{2\pi k}{7},0)$ along the bottom edge of this square, and 7 equally spaced but ``offset" points of the form $(\frac{2\pi k+\pi}{7},\pi)$ along the line $y=\pi$ in this square.
The dihedral group of order 14 will act by translations and glide reflections, permuting these 14 points.
We can similarly obtain some information relating the Vietoris--Rips complexes of the quotient space $X/G$ to the quotient of the Vietoris--Rips complex of the original space $X$, for $r$ small.
\end{example}


\begin{example}
Consider the 22-holed torus $X$ depicted in Figure~\ref{fig:torus} in $\R^3,$ with the Euclidean submetric.
Equip $X$ with the action of $G=\Z/7\Z$ as a group of rotations by $2\pi/7$.
The quotient space $X/G$ is obtained from a single severed arm of the torus after identifying two boundary circles: this is a 4-holed torus with an asymmetric metric.
For scale $r$ small, the Vietoris--Rips complex of the $X$, after quotienting this complex by $G$, is isomorphic to the Vietoris--Rips complex of the quotient 4-holed torus $X/G$.
Hatcher notes that these types of actions of $\Z/m\Z$ on an $(mn + 1)$-holed torus are the only covering space actions on this torus (Example~1.41 of~\cite{Hatcher}).
\end{example}

\begin{figure}[htb]
\centering
\includegraphics[width=0.55\textwidth]{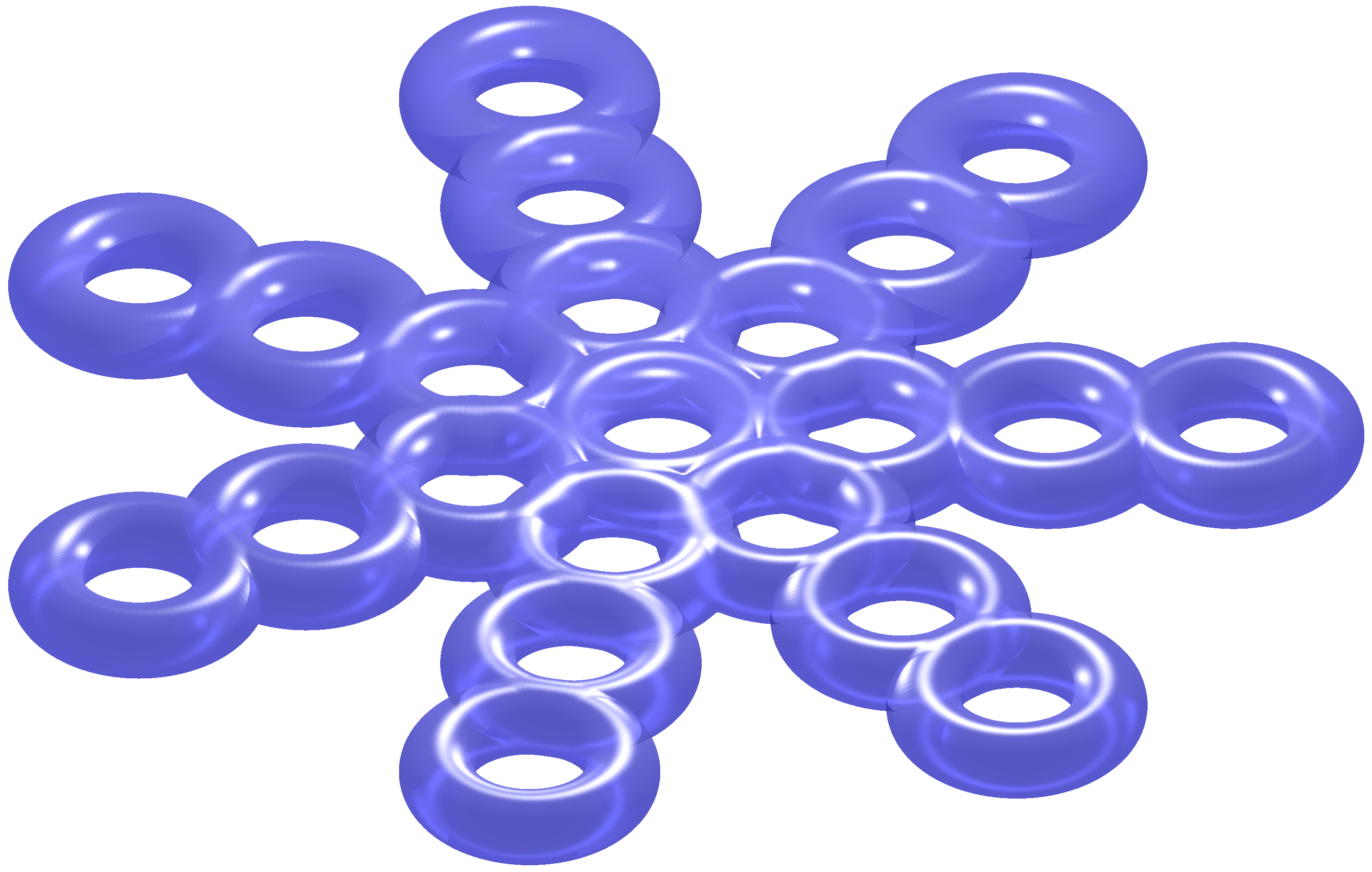}
\caption{A 22-holed torus in $\R^3$.}
\label{fig:torus}
\end{figure}

\begin{example}
Section~5 of the paper~\cite{Adamaszek2020} considers a space $X$ that is an ``infinite ladder'' with a countable number of rungs that is equipped with an action by the group of integers $G=\Z$ which is generated by translating by $n$ rungs.
The quotient space of this action $X/G$ is a ``circular ladder'' with $n$ rungs.
The paper~\cite{Adamaszek2020} uses our Proposition~\ref{prop:group-action} in order to understand the homotopy type of the Vietoris--Rips complex of the circular ladder $X/G$ in terms of the (known) Vietoris--Rips complex of the infinite ladder $X$.
\end{example}

In the following two sections we treat, with a fair amount of detail, the example where $X$ is the $n$-sphere, equipped with its antipodal action, and hence $X/G$ is real projective $n$-space.

\section{Vietoris--Rips thickenings of projective spaces at small scales}\label{sec:small}

We study Vietoris--Rips thickenings of projective spaces at small scales in this section before proceeding to large scales in the following section.

\subsection{Notation}

We first describe our notation for spheres and projective spaces.

\subsection*{Spheres}
The $n$-sphere $S^n$ is the set of points at distance one from the origin in Euclidean space:
\[S^n = \left\lbrace (x_1, x_2, x_3, \ldots, x_{n+1})\in\R^{n+1}~|~\sum_i x_i^2 = 1 \right\rbrace.\]
A metric on $S^n$ may be defined by either retaining the Euclidean metric by viewing $S^n$ as a subspace of $\mathbb{R}^{n+1}$, or by looking at the geodesic arc length of walking along $S^n$ as a surface.
For convenience, we will equip $S^n$ with the geodesic metric where the circumference of any great circle is one.

\subsection*{Real projective space}

The projective space $\RP^n$ is the quotient space
\[\RP^n = S^n / \sim\]
where the $\sim$ relation maps vectors to other vectors through multiplication by $\pm 1$.
We often denote a point $\pm x\in \RP^n$ with the notation $[x]$.
We will equip $\RP^n$ with the quotient metric induced from the geodesic metric on $S^n$.
That is, we have
\[d_{\RP^n}([x],[x'])=\min \{d_{S^n}(x,x'),d_{S^n}(x,-x')\}.\]
This is a specific case of the metric on a quotient spaces defined in Eq.~\eqref{eq:quotient-metric}.
Since each great circle in $S^n$ has circumference one, the ``great circles" in $\RP^n$ have circumference $\frac{1}{2}$.
One could instead use the Euclidean metric on $S^n$ to get a different metric on $\RP^n$; our results also apply to this case with only minor changes to the relevant scale parameters.

\subsection{Complexes and thickenings}

For $r<\frac{1}{6}$, we will show that $\vrm{\RP^n}{r}\simeq \RP^n$.
We do this by noting that $\RP^n=S^n/(x\sim-x)$ is the quotient of $S^n$ under the action of $G = (\lbrace \pm 1\rbrace, \times)\cong\Z/2\Z$.
The following two lemmas imply in Corollary~\ref{sphere-diam} that this action is a $\frac{1}{6}$-diameter action (though the constant in the first lemma is slightly better).

We note that $\frac{1}{6}$ is twice the filling radius of this projective space (with diameter $\frac{1}{4}$)~\cite{katz1983filling}, and hence the homotopy equivalence $\vrless{\RP^n}{r}\simeq \RP^n$ for $r<\frac{1}{6}$ given by Theorem~\ref{thm:small-scales} is also closely related to the recent preprint~\cite{lim2020vietoris}.
This perspective does not provide the first new homotopy type given by Theorem~\ref{thm:large-scales} in Section~\ref{sec:large}, to our knowledge.

\begin{lemma}\label{lem:diam}
Suppose $x_0,\ldots,x_k\in S^n$ with $\diam\{x_0,\ldots,x_k\}<\frac{1}{4}$ and let $g_i \in \lbrace \pm 1\rbrace$ for all $0\le i\le k$.
Then $\diam\{g_0x_0,\ldots, g_kx_k\}<\frac{1}{4}$ if and only if all signs $g_i$ are chosen to be positive or all signs are chosen to be negative.
\end{lemma}

\begin{proof}
Suppose $d(x_i,x_j)<\frac{1}{4}$.
Then we compute
\begin{itemize}
\item $d(-x_i,-x_j)=d(x_i,x_j)<\frac{1}{4}$,
\item $d(x_i,-x_j)=\frac{1}{2}-d(x_i,x_j)>\frac{1}{4}$, and
\item $d(-x_i,x_j)=\frac{1}{2}-d(x_i,x_j)>\frac{1}{4}$.
\end{itemize}
This means that $\diam\{g_i x_i, g_j x_j\}<\frac{1}{4}$ if and only if the signs $g_i$ and $g_j$ are both positive or both negative, from which the claim follows.
\end{proof}

\begin{lemma}\label{lem:diam2}
If $\diam_{\RP^n}\{[x_0],\ldots,[x_k]\}<\frac{1}{6}$, then there is a choice of signs $g_i \in \lbrace \pm 1\rbrace$ for $1\le i\le k$ such that
\[\diam_{S^n}\{x_0,g_1x_1\ldots,g_kx_k\}<\tfrac{1}{6}.\]
\end{lemma}

\begin{proof}
Let $1\le i\le k$.
By hypothesis, we have $d_{\RP^n}([x_0],[x_i])<\frac{1}{6}$, and, hence, we can pick a sign $g_i$ such that $d_{S^n}(x_0,g_ix_i)<\frac{1}{6}$.
For the remainder of this proof, we let $x_i^*$ denote $g_ix_i$.

We have chosen signs such that $d_{S^n}(x_0,x_i^*)<\frac{1}{6}$ for all $i$; it remains to show that $d_{S^n}(x_i^*,x_j^*)<\frac{1}{6}$ for all $1\le i,j\le k$.
If not, then since $d_{\RP^n}([x_i],[x_j])<\frac{1}{6}$, necessarily we would have that $d_{S^n}(x_i^*,-x_j^*)<\frac{1}{6}$.
However, since $d_{S^n}(-x_i^*,-x_0)=d_{S^n}(x_j^*,x_0)<\frac{1}{6}$, this would give the contradiction
\[\tfrac{1}{2}=d_{S^n}(x_0,-x_0)\le d_{S^n}(x_0,x_i^*) + d_{S^n}(x_i^*,-x_j^*) + d_{S^n}(-x_j^*,-x_0) < 3\cdot\tfrac{1}{6}=\tfrac{1}{2},\]
from the triangle inequality.
Hence, it must be the case that $d_{S^n}(x_i^*,x_j^*)<\frac{1}{6}$ for all $1\le i,j\le k$, and, therefore, the action of $G$ on $S^n$ is a $\frac{1}{6}$-diameter action.
\end{proof}

Qualitatively, we see from Lemma~\ref{lem:diam2} that a cluster of sufficiently close points in $\RP^n$ has as its preimage ``two clusters" of sufficiently close points in $S^n$.
The above two lemmas combine together to give the following.

\begin{corollary}\label{sphere-diam}
The action of $G = (\lbrace \pm 1\rbrace, \times) \cong \Z/2\Z$ on $S^n$ for $n\ge 1$ is an $r$-diameter action for $r<\frac{1}{6}$, and this bound is tight.
\end{corollary}

\begin{proof}
The existence part of the definition of an $r$-diameter action is given by Lemma~\ref{lem:diam2}, and uniqueness is given by Lemma~\ref{lem:diam}.

Furthermore, the bound $r<\frac{1}{6}$ in Lemma~\ref{lem:diam2} is tight.
Indeed, if we had $r=\frac{1}{6}$, then a counterexample is obtained by letting $\{[x_0],[x_1],[x_k]\}$ be three evenly-spaced points on a great circle of $\RP^n$ for $n\ge 1$ whose preimage is six evenly-spaced points on a great circle of $S^n$.
\end{proof}

We are now prepared to study Vietoris--Rips thickenings of projective spaces at sufficiently small scale parameters.
Let $\sim$ denote the equivalence relation on $\vrm{S^n}{r}$ induced by the canonical $\Z/2\Z$ action on $S^n$; more explicitly this equivalence relation is given by $\sum_i \lambda_i x_i\sim\sum_i \lambda_i (-x_i)$.
By definition $\RP^n\cong (S^n/\sim)$; Lemma~\ref{lem:simplicial-iso} will generalize this to say that for $r$ sufficiently small, we also have $\vrm{\RP^n}{r}\cong\vrm{S^n}{r}/\sim$.

Let $W$ be the set of all interior points of equilateral 2-simplices in $\vrleq{\RP^n}{\frac{1}{6}}$ inscribed in a great circle of $\RP^n$.
More precisely,
\[W = \left\{\sum_{i=0}^2 \lambda_i x_i ~\Big|~\lambda_i>0\mbox{ and }\{[x_0],[x_1],[x_2]\}\mbox{ is a regular 2-simplex in a great circle}\right\}.\]

\begin{lemma}\label{lem:simplicial-iso}
We have have the following isomorphisms of simplicial complexes and homeomorphisms of metric thickenings.
\begin{align*}
\vr{\RP^n}{r}&\cong\vr{S^n}{r}/\sim \mbox{ for }r<\tfrac{1}{6}\\
\vrm{\RP^n}{r}&\cong\vrm{S^n}{r}/\sim \mbox{ for }r<\tfrac{1}{6}\\
\vrless{\RP^n}{\tfrac{1}{6}}&\cong\vrless{S^n}{\tfrac{1}{6}}/\sim\\
\vrmless{\RP^n}{\tfrac{1}{6}}&\cong\vrmless{S^n}{\tfrac{1}{6}}/\sim\\
\vrleq{\RP^n}{\tfrac{1}{6}}\setminus W&\cong\vrleq{S^n}{\tfrac{1}{6}}/\sim\\
\vrmleq{\RP^n}{\tfrac{1}{6}}\setminus W&\cong\vrmleq{S^n}{\tfrac{1}{6}}/\sim
\end{align*}
\end{lemma}

\begin{proof}
We prove the case of the simplicial complexes; the proof for metric thickenings is analogous.

We first consider the cases of $\vr{\RP^n}{r}$ with $r<\tfrac{1}{6}$, and $\vrless{\RP^n}{\frac{1}{6}}$.
The group $\mathbb{Z}/2\mathbb{Z}$ acts properly by isometries on $S^n$ and by Corollary~\ref{sphere-diam} is a $\frac{1}{6}$-diameter action.
Thus, by Proposition~\ref{prop:group-action}, we have $\vr{\RP^n}{r}=\vr{S^n/\sim}{r}$ is homeomorphic to $\vr{S^n}{r}/\sim$.

The case of $\vrleq{\RP^n}{\tfrac{1}{6}}\setminus W$ follows similarly.
Indeed, since $W$ has been removed, for any simplex $\{[x_0],\ldots,[x_k]\} \in \vrleq{\RP^n}{\tfrac{1}{6}}\setminus W$, there exists a unique choice of elements $g_i\in (\{\pm 1\},\times)$ for $1\le i\le k$ such that $\diam_{S^n}\{x_0,g_1\cdot x_1\ldots,g_k\cdot x_k\}=\diam_{\RP^n}\{[x_0],\ldots,[x_k]\}$.
\end{proof}

\subsection{Thickenings}

We now identify the scale parameters which are sufficiently small so that the Vietoris--Rips thickenings of the projective space $\RP^n$ are homotopy equivalent to $\RP^n$.
Though Hausmann's theorem~\cite{Hausmann1995,AAF} guarantees that such a sufficiently small scale parameter exists, we identify the optimal such scale because the bounds given in Hausmann's theorem are not optimal.
We restrict attention to metric thickenings instead of simplicial complexes, as those are the results needed in Section~\ref{sec:large} in order to study larger scales.

Let $f\colon S^n\to\R^{n+1}$ be the inclusion map, and extend linearly to obtain a (non-injective) map $f \colon \vrm{S^n}{r}\to\R^{n+1}$ sending a formal convex combination of points in $S^n$ to its corresponding linear combination in $\R^{n+1}$.
Let $\pi\colon \R^{n+1}\setminus\{\vec{0}\}\to S^n$ be the radial projection map.
Let $r_n$ be the diameter of an inscribed regular $(n+1)$-simplex in $S^n$.
For $r<r_n$, the image of $f \colon \vrm{S^n}{r}\to\R^{n+1}$ misses the origin in $\R^{n+1}$ by the proof of Lemma~3 in~\cite{lovasz1983self}.
Hence, we have a composite map $\pi f\colon \vrm{S^n}{r}\to S^n$.
If furthermore $r<\frac{1}{6}$, then by Lemma~\ref{lem:simplicial-iso} we get an induced map
\[\vrm{\RP^n}{r}\cong(\vrm{S^n}{r}/\sim) \xrightarrow{f/\sim} ((\R^{n+1}\setminus\{\vec{0}\})/\sim) \xrightarrow{\pi/\sim} (S^n/\sim) = \RP^n.\]
By an abuse of notation, we also denote the above composite map by $\pi f\colon \vrm{\RP^n}{r}\to\RP^n$.

\begin{figure}[htb]
\centering
\includegraphics[width=0.27\textwidth]{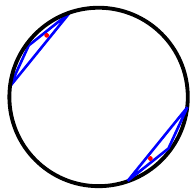}
\caption{A 2-simplex inscribed in $S^n$ (in the drawing $n=1$), along with its antipode, containing example points (drawn in red) that get identified under the map $\vrm{S^n}{r} \to (\R^{n+1}\setminus\{\vec{0}\})/\sim.$}
\label{fig:BorsukUlamRP1}
\end{figure}

\begin{theorem}\label{thm:small-scales}
The maps 
\begin{align*}
&\pi f\colon\vrm{\RP^n}{r}\to\RP^n \quad\mbox{for }r<\tfrac{1}{6}\\
&\pi f\colon\vrmless{\RP^n}{\tfrac{1}{6}}\to\RP^n \\
&\pi f\colon\vrmleq{\RP^n}{\tfrac{1}{6}}\setminus W\to\RP^n
\end{align*}
are homotopy equivalences.
\end{theorem}

We remark that the above theorem is also true for Vietoris--Rips simplicial complexes after adding the additional restriction that $r>0$.

\begin{proof}
We have that $\pi f\colon\vrm{S^n}{r}\to S^n$ is a homotopy equivalence by~\cite{AAF} since $r<r_n$.
Since $\pi f\colon\vrm{S^n}{r}\to S^n$ respects the identifications $\sim$, this gives that the three maps above are also homotopy equivalences.
\end{proof}

The following lemma will be used in the proof of Theorem~\ref{thm:large-scales} on the Vietoris--Rips thickenings of projective spaces at large scales.

\begin{lemma}\label{lem:bijective}
Let $\Delta$ be an equilateral 2-simplex inscribed in a great circle of $\RP^n$.
Note $\partial\Delta\subseteq\vrmleq{\RP^n}{\frac{1}{6}}\setminus W$.
The map $\pi f\colon \vrmleq{\RP^n}{\frac{1}{6}}\to \RP^n$, when restricted to $\partial \Delta$, is bijective onto its image, namely the great circle in $\RP^n$ in which $\Delta$ is inscribed.
\end{lemma}

\begin{proof}
We will work in the unit sphere $S^n\subseteq\R^{n+1}$, which is a double cover of $\RP^n$.
Without loss of generality, the equilateral triangle $\Delta$ can be supposed to have coordinates $\pm(1,0,\ldots,0)$, $\pm(\cos{\frac{\pi}{3}},\sin{\frac{\pi}{3}},0,\ldots,0)$, and $\pm(\cos{\frac{\pi}{3}},-\sin{\frac{\pi}{3}},0,\ldots,0)$.
That is, the triangle inscribed in $\RP^n$ can be viewed as a hexagon inscribed in $S^n$.
Considering $\pi f$ to have domain $\vr{S^n}{\frac{1}{6}}$ at first, we note that the restriction of $\pi f$ to this hexagon maps bijectively onto the great circle in $S^n$ given by $(\cos\theta,\sin\theta, 0, 0,\ldots 0)$, where $\theta \in [0, 2\pi)$, with antipodal points on the hexagon mapped to antipodal points on $S^n$.
After quotienting out by the antipodal action on both the domain and codomain, i.e.\ after returning to the point of view where $\pi f$ has domain $\vr{\RP^n}{\frac{1}{6}}$, we see that $\pi f|_{\partial \Delta}\colon \partial \Delta\to \RP^n$ maps bijectively onto its image, the great circle in $\RP^n$ in which $\Delta$ is inscribed.
\end{proof}

\section{Vietoris--Rips thickenings of projective spaces at large scales}\label{sec:large}

As a subgoal of this document, we would like to identify the homotopy type of $\vrmleq{\RP^n}{r}$, the Vietoris--Rips thickenings of projective space, for larger scale parameters $r$.
We are able to describe this homotopy type for $r=\frac{1}{6}$, which is the first scale parameter where the homotopy type of $\vrmleq{\RP^n}{r}$ changes.
The proof is analogous to (but more complicated than) the proof of the homotopy type of Vietoris--Rips thickenings of the sphere in Theorem~5.4 of~\cite{AAF}.
In Section~\ref{ss:sphere-large} we recall the proof of Theorem~5.4 of~\cite{AAF}, with a few more details added, so that we can set notation and clarify the ideas.
We then modify these techniques to handle the equivariant setting of the projective space in Section~\ref{ss:projective-large}.

At larger scales, we will restrict attention to the Vietoris--Rips metric thickenings and not discuss simplicial complexes.
The reason for this is as follows.
If $S^1$ is the geodesic circle of unit circumference, then $\vrmleq{S^1}{\frac{1}{3}}\simeq S^3$ is a 3-sphere~\cite{AAF,ABF}, whereas $\vrleq{S^1}{\frac{1}{3}}\simeq\bigvee^{\mathfrak{c}}S^2$ is an uncountably infinite wedge sum of 2-spheres~\cite{AA-VRS1}.
We think of the former homotopy type as being ``correct,'' and by contrast we think of the wild homotopy type of the simplicial complex $\vrleq{S^1}{\frac{1}{3}}$ as an artifact of the fact that it is equipped with the ``wrong'' topology.
Indeed, the topology on $\vrleq{S^1}{\frac{1}{3}}$ is such that the inclusion $S^1\hookrightarrow \vr{S^1}{\frac{1}{3}}$ is not even continuous, since the vertex set of a simplicial complex is equipped with the discrete metric.
Additionally, one should think of the 3-sphere $S^3$ as being the ``right'' homotopy type at scale $\frac{1}{3}$ since for all $0<\varepsilon<\frac{1}{15}$, we have $\vr{S^1}{\frac{1}{3}+\varepsilon}\simeq S^3$.
A similar story is true for the Vietoris--Rips thickenings and simplicial complexes of $n$-spheres and projective spaces, and this is why we now restrict attention to Vietoris--Rips metric thickenings.

\subsection{Vietoris--Rips thickenings of the sphere at large scales}\label{ss:sphere-large}

Let $S^n$ be the $n$-sphere equipped with the geodesic metric.
Recall $r_n$ is the diameter of an inscribed regular $(n+1)$-simplex in $S^n$.
For $r<r_n$, we have $\vrm{S^n}{r}\simeq S^n$.
The first new homotopy type of the Vietoris--Rips thickening of the sphere $\vrmleq{S^n}{r}$ is determined in Theorem~5.4 of~\cite{AAF} to be the join $S^n * \tfrac{\so(n+1)}{A_{n+2}}$.
We first obtain a homotopy equivalence to an adjunction space
\[\vrmleq{S^n}{r_n}\simeq S^n\cup_h\Biggl(D^{n+1}\times\frac{\so(n+1)}{A_{n+2}}\Biggr),\]
where $D^{n+1}$ is the closed $(n+1)$-dimensional ball, where $\frac{\so(n+1)}{A_{n+1}}$ parametrizes all regular oriented $(n+1)$-simplices $\Delta^{n+1}$ inscribed in $S^n$, and where $h\colon S^n\times\frac{\so(n+1)}{A_{n+1}}\to S^n$ via $h(x,y)=x$.

\begin{figure}[htb]
\centering
\includegraphics[width=0.27\textwidth]{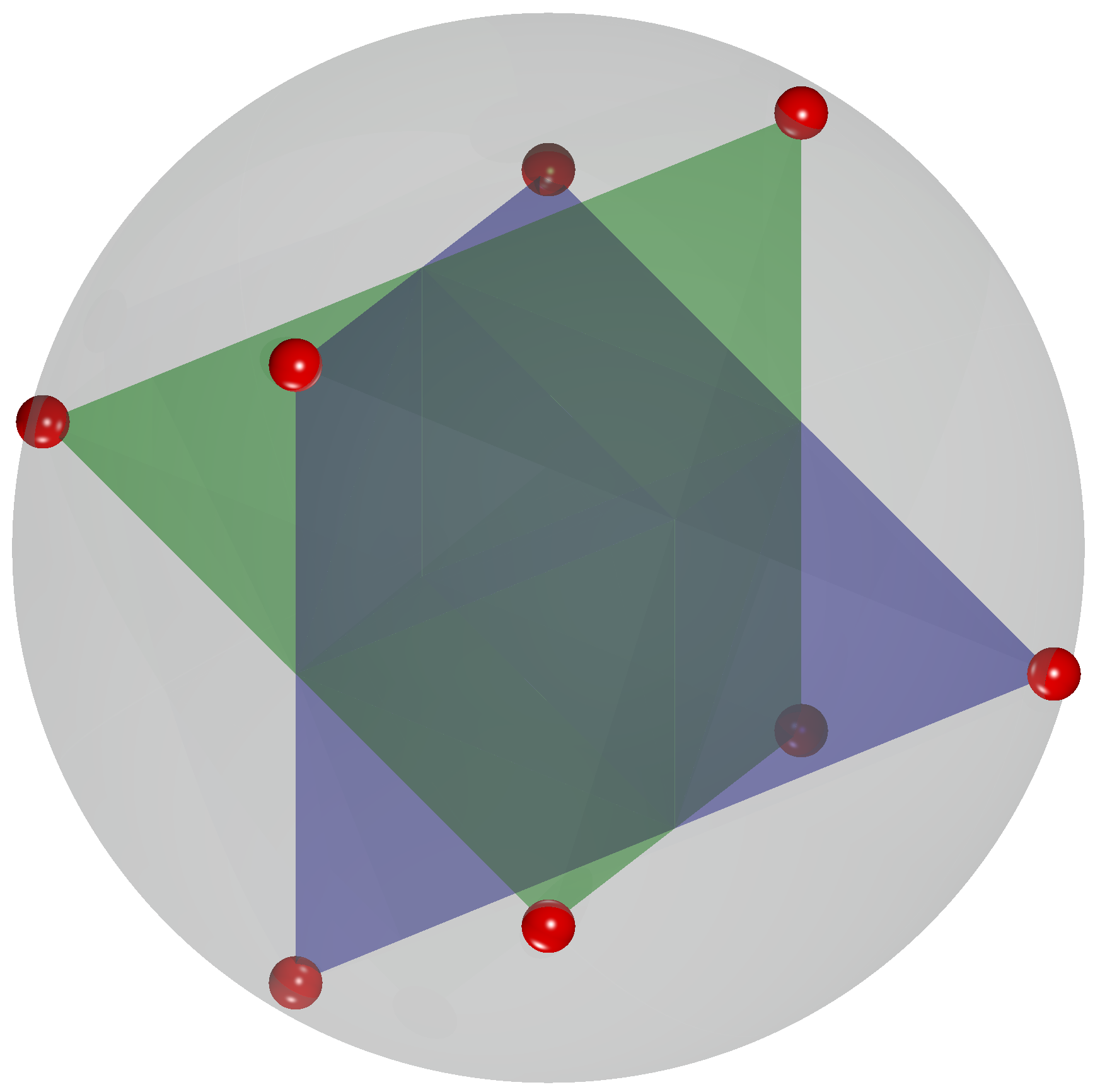}
\caption{Two regular inscribed simplices in $S^2$.}
\label{fig:tetrahedra}
\end{figure}

It takes some care to describe the parameter space $\frac{\so(n+1)}{A_{n+2}}$.
Note $\so(n+1)$ is a topological group, and as we explain in the paragraph below, the alternating group $A_{n+2}$ can be seen as a subgroup of $\so(n+1)$ even though there is no canonical way to do this.
Once $A_{n+2}$ has been identified with a subgroup of $\so(n+1)$, then $A_{n+2}$ acts on $\so(n+1)$ via left multiplication, as explained in Example~3.88(f) of~\cite{lee2010introduction}.
Therefore, we define $\frac{\so(n+1)}{A_{n+2}}$ as the quotient space or ``orbit space'' of this action, i.e.\
$\frac{\so(n+1)}{A_{n+2}}$ is $\so(n+1)/\sim$, where $\sim$ is the equivalence relation where $x \sim g\cdot x$ for all $x \in \so(n+1)$ and $g \in A_{n+2}$.
We emphasize that we are not trying to identify $A_{n+2}$ with a normal subgroup of $\so(n+1)$, nor are we trying to give $\frac{\so(n+1)}{A_{n+2}}$ the structure of a quotient group.

The above paragraph relies on identifying the alternating group $A_{n+2}$ with a subgroup of $\so(n+1)$.
To do this, fix a regular $(n+1)$-dimensional simplex inscribed in $S^n$ inside $\R^{n+1}$, with the center of the simplex at the origin.
The $(n+1)$-simplex has $n+2$ vertices, and $A_{n+2}$ as its group of rotational symmetries.
We can therefore associate each element $g\in A_{n+2}$ with an $(n+1)\times(n+1)$ rotation matrix that permutes the vertices of the simplex in the same way that $g$ does.
For example, if $n=1$, then no matter what fixed regular 2-simplex inscribed in $S^1$ in $\R^2$ one picks, the three elements of $A_3 = \Z/3\Z$ will be the rotation matrices
\[
\begin{bmatrix}1 & 0 \\ 0 & 1\end{bmatrix}, 
\begin{bmatrix}\cos \frac{2\pi}{3} & -\sin \frac{2\pi}{3} \\ \sin \frac{2\pi}{3} & \cos \frac{2\pi}{3}\end{bmatrix}, \text{ and }
\begin{bmatrix}\cos \frac{4\pi}{3} & -\sin \frac{4\pi}{3} \\ \sin \frac{4\pi}{3} & \cos \frac{4\pi}{3} \end{bmatrix}.
\]
However, if $n\ge 2$, then different choices of a fixed regular $(n+1)$-simplex inscribed in $S^n$ in $\R^{n+1}$ will give different rotation matrices corresponding to the elements of $A_{n+2}$.
Since any two ways of viewing $A_{n+2}$ as a subgroup of $\so(n+1)$ are conjugate, Exercise~24(b) in Section~1.3 of~\cite{Hatcher} implies that the homeomorphism type of $\frac{\so(n+1)}{A_{n+2}}$ does not depend on this choice.


\begin{theorem}[Theorem~5.4 of~\cite{AAF}]\label{thm:S^n-critical}
We have a homotopy equivalence 
\[\vrmleq{S^n}{r_n}\simeq S^n * \tfrac{\so(n+1)}{A_{n+2}}.\]
\end{theorem}

\begin{proof}
Let $W$ be the set of all interior points of regular $(n+1)$-simplices inscribed in $\vrleq{S^n}{r_n}$.
More precisely,
\[W = \bigl\{\textstyle{\sum_{i=0}^{n+1}}\lambda_i x_i~\big|~\lambda_i>0\mbox{ for all }i\mbox{ and }\{x_0,\ldots,x_{n+1}\}\mbox{ is a regular }(n+1)\mbox{-simplex}\bigr\}.\]
Note the closure of $W$ in $\vrleq{S^n}{r_n}$ is homeomorphic to $D^{n+1} \times \tfrac{\so(n+1)}{A_{n+2}}$.
We construct the following commutative diagram.
\begin{center}
\begin{tikzpicture}[description/.style={fill=white,inner sep=2pt}] 
\matrix (m) [matrix of math nodes, row sep=3em, 
column sep=4em, text height=1.5ex, text depth=0.25ex] 
{  
& S^n\\
D^{n+1} \times \frac{\so(n+1)}{A_{n+2}} \supseteq S^n \times \frac{\so(n+1)}{A_{n+2}} & \\
& \vrmleq{S^n}{r_n}\setminus W\\
};
\path[->,font=\scriptsize]
(m-2-1) edge node[above] {$h$} (m-1-2)
(m-2-1) edge node[auto] {$g$} (m-3-2)
(m-3-2) edge node[auto] {$\pi f$} (m-1-2)
;
\path[,font=\scriptsize]
(m-3-2) edge node[right] {$\simeq$} (m-1-2)
;
\end{tikzpicture}
\end{center}
The function $h\colon S^n \times \frac{\so(n+1)}{A_{n+2}}\to S^n$ is defined by $h(x,y)=x$.
For $y \in \tfrac{\so(n+1)}{A_{n+2}}$, let $\{y_0,\ldots,y_{n+1}\}$ be the $n+2$ vertices of the rotated regular $(n+1)$-simplex parameterized by $y$.
Let
\[\partial\Delta_y=\Biggl\{\sum_{i=0}^{n+1} \lambda_i y_i\in\vrmleq{S^n}{r_n}\setminus W~\Big|~\lambda_i=0\mbox{ for some }i\Biggr\}\] 
be the boundary of the corresponding simplex.
Note $\pi f|_{\partial\Delta_y} \colon \partial\Delta_y \to S^n$ is bijective.
Define map $g\colon S^n \times \frac{\so(n+1)}{A_{n+2}}\to\vrmleq{S^n}{r_n}\setminus W$ by letting $g(x,y)$ be the unique point of $\partial\Delta_y$ such that $\pi f(g(x,y))=x$; that is, $g(x,y)=(\pi f|_{\partial\Delta_y})^{-1}(x)$.
We have $\pi f\circ g=h$, meaning the square commutes.

We now have the following sequence of homotopy equivalences, where $C(X)$ denotes the cone of a topological space $X$.
\begin{align*}
\vrmleq{S^n}{r_n} &= \bigl((\vrmleq{S^n}{r_n}\setminus W\bigr) \cup_g \bigl(D^{n+1} \times \tfrac{\so(n+1)}{A_{n+2}}\bigr) \\
&\simeq S^n \cup_h \bigl(D^{n+1} \times \tfrac{\so(n+1)}{A_{n+2}}\bigr) \\
&\simeq \Bigl(S^n \times C\bigl(\tfrac{\so(n+1)}{A_{n+2}}\bigr)\Bigr) \cup_{S^n \times \tfrac{\so(n+1)}{A_{n+2}}} \Bigr(C(S^n) \times \tfrac{\so(n+1)}{A_{n+2}} \Bigl)\\
&= S^n * \tfrac{\so(n+1)}{A_{n+2}}.
\end{align*}
Indeed, the first line is by the definitions of $W$, of $g$, and of adjunction spaces.
The second line follows from the commutative diagram above and the homotopy invariance properties of adjunction spaces (7.5.7 of~\cite{brown2006topology} or Proposition~5.3.3 of~\cite{tom2008algebraic}).
The third line follows from these same properties of adjunction spaces, induced by contractibility of $C(\tfrac{\so(n+1)}{A_{n+2}})$.
The fourth line uses an equivalent definition for the join of two topological spaces as $Y*Z=Y\times C(Z)\cup_{Y\times Z}C(Y)\times Z$.
\end{proof}

\subsection{Vietoris--Rips thickenings of the real projective plane at large scales}\label{ss:projective-large}

We henceforth have proved that the metric thickening $\vrmleq{\RP^n}{r}$ is homotopy equivalent to $\RP^n$ for $r < \frac{1}{6}$.
The first change in homotopy type occurs when $r = \frac{1}{6}$.
Indeed, for $n = 1$, we have $\vrmleq{\RP^n}{\frac{1}{6}} = S^3$, as there is a homeomorphism between $\RP^1$ and $S^1$ which results in an isometry (up to scaling all distances by two) between $\vrmleq{\RP^1}{\frac{1}{6}}$ and $\vrmleq{S^1}{\frac{1}{3}}$.

The flavor of $\vrm{\RP^n}{r}$ is different than that of $\vrm{S^n}{r}$ at large scales $r$.
Indeed, whereas the homotopy type of $\vrm{S^n}{r}$ first changes due to the appearance of regular $(n+1)$-simplices~\cite{AAF},
the homotopy type of $\vrm{\RP^n}{r}$ first changes due to the appearance of (lower-dimensional) 2-simplices inscribed in great circles of $\RP^n$.

\begin{theorem}\label{thm:large-scales}
The metric space $\vrmleq{\RP^n}{\frac{1}{6}}$ has the homotopy type of a $(2n+1)$-dimensional CW complex.
\end{theorem}

\begin{proof}
Let $\Gr(k,d)$ denote the Grassmannian of all $k$-planes through the origin in $\R^d$.
The space $\Gr(k,d)$ is a manifold of dimension $k(d-k)$.
Let $I=[0,1]$ be the closed unit interval.

We define
\[Y=\{(V,\pm y,\pm x, r)\in \Gr(2,n+1)\times \RP^n\times \RP^n\times I~| \pm y,\pm x\in V\}/\sim,\] 
where $\sim$ will be defined below.
The 2-plane $V\in \Gr(2,n+1)$ encodes a great circle in $\RP^n$, i.e., the 2-fold quotient of intersection circle of $V$ with $S^n\subseteq \R^{n+1}$.
The point $\pm y$ encodes a point along that great circle, the point $\pm x$ encodes a second point along that great circle, and the radius $r\in I$ encodes a radius inside a disk.
The identifications $\sim$ are defined as follows.
\begin{itemize}
\item $(V,\pm y,\pm x,0)\sim(V,\pm y,\pm x',0)$ for all $x$ and $x'$.
\item $(V,\pm y,\pm x, r)\sim(V,\pm y',\pm x,r)$ for any points $\pm y$ and $\pm y'$ whose angles in the great circle corresponding to $V$ are a multiple of $\frac{2\pi}{6}$ apart.
\end{itemize}
The point $(V,\pm y,\pm x, r)$ can be thought of as a point of radius $r$ at angle $\pm x$ in a disc attached to the great circle corresponding to $V$, where the boundary of that disk will be attached to $\vrmleq{\RP^n}{\frac{1}{6}}\setminus W$ via some map $g$ (defined below) along an equilateral triangle containing $\pm y$.
Indeed, observe that if $V$ and $y$ are fixed, then $\{(V, \pm y,\pm x,r)\subseteq Y~|~V=V_0,y=y_0\text{ are fixed}\}$ is homeomorphic to a disk.
The first bullet point defining $\sim$ above is since in polar coordinates, the center of any disc has radius $r=0$ and an undetermined angle that could correspond to any $\pm x$.
The second bullet point defining $\sim$ above is so that inscribed triangles with a vertex at $y$ or $y'$ (whose angles in the great circle corresponding to $V$ are a multiple of $\frac{2\pi}{6}$ apart) are identified.

We let $Z\subseteq Y$ be the subset of all points of the form $(V,\pm y,\pm x, 1)$, i.e., those points that are on some great circle.
Consider the following commutative diagram, where the vertical map is a homotopy equivalence by Theorem~\ref{thm:small-scales}.

\begin{center}
\begin{tikzpicture}[description/.style={fill=white,inner sep=2pt}] 
\matrix (m) [matrix of math nodes, row sep=3em, 
column sep=4em, text height=1.5ex, text depth=0.25ex] 
{  
& \RP^n\\
Y \supseteq Z & \\
& \vrmleq{\RP^n}{\tfrac{1}{6}}\setminus W\\
};
\path[->,font=\scriptsize]
(m-2-1) edge node[above] {$h$} (m-1-2)
(m-2-1) edge node[auto] {$g$} (m-3-2)
(m-3-2) edge node[auto] {$\pi f$} (m-1-2)
;
\path[,font=\scriptsize]
(m-3-2) edge node[right] {$\simeq$} (m-1-2)
;
\end{tikzpicture}
\end{center}

We define map $h\colon Z\to \RP^n$ by $h(V,\pm y,\pm x, 1)=\pm x$.
We define $g\colon Z\to\vrmleq{\RP^n}{\frac{1}{6}}\setminus W$ as follows.
Let $\Delta$ be the equilateral triangle containing $\pm y$ that is inscribed in the great circle corresponding to $V$.
Define $g(V, \pm y, \pm x, 1)$ to be the unique point on the 1-skeleton $\partial \Delta$ of this 2-simplex such that $\pi f(g(V, \pm y, \pm x, 1))=\pm x$; existence and uniqueness of this point follow since $\pi f|_{\partial \Delta}\colon \partial \Delta\to \RP^n$ is bijective onto its image by Lemma~\ref{lem:bijective}.
It follows that $\pi f \circ g=h$.
Therefore, we have the following homotopy equivalence.
\[
\vrmleq{\RP^n}{\tfrac{1}{6}} = \bigl(\vrmleq{\RP^n}{\tfrac{1}{6}}\setminus W\bigr) \cup_g Y \simeq \RP^n \cup_h Y,
\]
where the last step is by Theorem~\ref{thm:small-scales} and the homotopy invariance properties of adjunction spaces (7.5.7 of~\cite{brown2006topology} or Proposition~5.3.3 of~\cite{tom2008algebraic}).

It remains to show that $\RP^n \cup_h Y$ is homotopy equivalent to a $(2n+1)$-dimensional CW complex.
We begin with the torus bundle
\[T' = \{(V,\pm y,\pm x)\in\Gr(2,n+1)\times \RP^n\times \RP^n~|~\pm y,\pm x\in V\}\]
over $\Gr(2,n+1)$, with projection map $T'\to\Gr(2,n+1)$ via $(V,\pm y,\pm x)\mapsto V$.
Consider also the circle bundle
\[C' = \{(V,\pm y)\in\Gr(2,n+1)\times \RP^n~|~\pm y\in V\}\]
over $\Gr(2,n+1)$, with projection map $C'\to\Gr(2,n+1)$ via $(V,\pm y)\mapsto V$.
The space $T:=T'/\sim_1$, where $(V,\pm y,\pm x)\sim_1(V,\pm y',\pm x)$ for any points $\pm y$ and $\pm y'$ whose angles in the great circle corresponding to $V$ are a multiple of $\frac{2\pi}{6}$ apart, is also a torus bundle over $\Gr(2,n+1)$.
Hence $T$ is a manifold of dimension two more than $\dim(\Gr(2,n+1))=2(n-1)$, meaning $\dim(T)=2n$.
Similarly, the space $C:=C'/\sim_1$, where $(V,\pm y)\sim_1(V,\pm y')$ is defined analogously, is a circle bundle over $\Gr(2,n+1)$.
The space $T\times I$ is therefore a $(2n+1)$-dimensional manifold with boundary, and hence a $(2n+1)$-dimensional CW complex. Finally we claim that $Y=(T\times I)/\sim_2$, where $(V,\pm y,\pm x,0)\sim_2(V,\pm y,\pm x',0)$ for all $x$ and $x'$, is also a $(2n+1)$-dimensional CW complex.
Indeed, note that $T\times\{0\}$ is a CW subcomplex of $T\times I$.
The map $q\colon T\times\{0\}\to C$ defined by $(V,\pm y,\pm x,0)\mapsto (V,\pm y)$ is a differentiable fiber bundle (with circular fibers).
Hence Corollary~2.2 of~\cite{putz1967triangulation} states that we can put simplicial complex structures on $T\times \{0\}$ and $C$ so that $q$ is simplicial; see also~\cite{verona1979triangulation}. 
Since $q$ is cellular, the adjunction space $C\cup_q (T\times I)$ is a CW complex, and so we we have that
\[ Y=((T\times I)/\sim_2) \cong C\cup_q (T\times I)\]
is a $(2n+1)$-dimensional CW complex.
To see that $Z$ is a CW subcomplex of $Y$, note that $Z$ sits inside $Y$ as $T\times \{1\}$.
It follows from Corollary~IV.2.5\footnote{Related results are Theorem~II.5.11 of~\cite{lundell2012topology} or Theorem~II.4.3 of~\cite{whitehead2012elements}, which furthermore implies that if $h$ is a cellular map, then $\RP^n\cup_h Y$ is a CW complex.} of~\cite{lundell2012topology} that the adjunction space $\RP^n \cup_h Y$ is homotopy equivalent to a $(2n+1)$-dimensional CW complex.
\end{proof}

We obtain as a consequence the following corollary.
We remark that this corollary is far from obvious, as $\vrmleq{\RP^n}{\frac{1}{6}}$ is in some sense ``infinite dimensional.''

\begin{corollary}
Since $\vrmleq{\RP^n}{\frac{1}{6}}$ has the homotopy type of a $(2n+1)$-dimensional CW complex, its homology and cohomology groups are trivial in dimensions $2n+2$ and larger.
\end{corollary}

\begin{conjecture}
We conjecture that there is some $\varepsilon>0$ sufficiently small such that for all $0<\delta<\varepsilon$, the homotopy types of $\vrm{\RP^n}{\frac{1}{6}+\delta}$ and $\vr{\RP^n}{\frac{1}{6}+\delta}$ are equal to that of $\vrmleq{\RP^n}{\frac{1}{6}}$.
\end{conjecture}

\begin{question}
What are the homotopy types of $\vrm{\RP^n}{r}$ at larger scale parameters $r>\frac{1}{6}$, and, in particular, what is the smallest value of $\varepsilon>0$ for which which we obtain a new homotopy type $\vrmleq{\RP^n}{\frac{1}{6}+\varepsilon}\not\simeq\vrmleq{\RP^n}{\frac{1}{6}}$?
\end{question}

\section{Conclusion}

We have initiated the study of what happens when a group acts on a metric space, and hence also on its Vietoris--Rips simplicial complex and metric thickening, at intermediate scale parameters.
We show that, for small enough scale parameters $r$, both the simplicial complex $\vr{X/G}{r}$ and the metric thickening $\vrm{X/G}{r}$ are homotopy equivalent to $\vr{X}{r}/G$ and $\vrm{X}{r}/G$, respectively.
We give precise quantitative control on which scale parameters $r$ are small enough and provide a similar result for \v{C}ech complexes.
We further extend these results to analyze the homotopy types of Vietoris--Rips thickenings of real projective spaces at the first scale parameter where their homotopy types change.

We end with a description of a few open questions motivated by this work.

\begin{question}
What are the homotopy types of the \v{C}ech complexes $\cech{\RP^n}{r}$ of projective spaces?
We note that the action of $G=(\{\pm1\},\times)\cong\Z/2\Z$ on $S^n$ is an $r$-nerve action for all $r<\frac{1}{8}$, where the circumference of a great circle in $S^n$ is 1 (and so the circumference of a great circle in $\RP^n$ is $\frac{1}{2}$).
Hence $\cech{\RP^n}{r}\simeq \RP^n$ for all $r<\frac{1}{8}$.
What are the homotopy types of $\cech{\RP^n}{r}$ at larger scales?
\end{question}

\begin{question}
We note that $\RP^3$ is just one example of a spherical 3-manifold, i.e., a quotient space $S^3/G$ where $G$  is a finite subgroup of $\so(4)$ acting freely by rotations.
What can one say about Vietoris--Rips thickenings of other spherical manifolds?
\end{question}

\begin{question}
In addition, what can be said about lens spaces?
Let $S^{2n-1}$ be the unit sphere in complex $n$-dimensional space $\mathbb{C}^n$.
For integers $p, \ell_1, \ldots, \ell_n$ with each $\ell_i$ relatively prime to $p$, we define the lens space $L(p;\ell_1,\ldots,\ell_n)$ to be the quotient of $S^{2n-1}$ under the action of $\Z/p\Z$ generated by
\[(x_1,\ldots,x_n)\mapsto(e^{2\pi \ell_1/p}x_1,\ldots,e^{2\pi \ell_n/p}x_n).\]
See Example~2.43 of~\cite{Hatcher}.
Any such lens space has fundamental group $\Z/p\Z$.
Interestingly, different choices of the $\ell_i$'s can produce lens spaces that are either homeomorphic, homotopy equivalent but not homeomorphic, or not homotopy equivalent.
What can be said about the homotopy types of Vietoris--Rips thickenings of lens spaces?
\end{question}

\section*{Acknowledgement}
We would like to thank Amit Patel for bringing the papers~\cite{putz1967triangulation,verona1979triangulation}, used in the proof of Theorem~\ref{thm:large-scales}, to our attention.
We would also like to thank the authors of~\cite{Adamaszek2020} for helpful conversations.
This material is based upon work supported by the National Science Foundation under Grants No.\ 1633830, 1712788, 1830676, and 1934725.

\nocite{virk2017approximations}

\bibliographystyle{plain}
\bibliography{MetricThickeningsAndGroupActions}

\end{document}